\documentclass[a4paper,twoside]{article}

\usepackage{epsfig}
\usepackage{subfigure}
\usepackage{calc}
\usepackage{amssymb}
\usepackage{amstext}
\usepackage{amsmath}
\usepackage{amsthm}
\usepackage{amsfonts}
\usepackage{multicol}
\usepackage{pslatex}
\usepackage{apalike}
\usepackage{cite}
\usepackage{graphicx}
\usepackage{SCITEPRESS}     % Please add other packages that you may need BEFORE the SCITEPRESS.sty package.

\newtheorem{theorem}{Theorem}[section]
\newtheorem{lemma}[theorem]{Lemma}

\newenvironment{definition}[1][Definition]{\begin{trivlist}
\item[\hskip \labelsep {\bfseries #1}]}{\end{trivlist}}
\begin{document}

\title{A New Family of Bounded Divergence Measures and Application to Signal Detection}

\author{\authorname{Shivakumar Jolad\sup{1}, Ahmed~Roman\sup{2}, Mahesh~C.~Shastry \sup{3}, Mihir~Gadgil\sup{4} and Ayanendranath~ Basu \sup{5}}
\affiliation{\sup{1}Department of Physics, Indian Institute of Technology  Gandhinagar,  Ahmedabad, Gujarat, INDIA}
\affiliation{\sup{2} Department of Mathematics, Virginia Tech , Blacksburg, VA, USA.}
\affiliation{\sup{3} Department of Physics, Indian Institute of Science Education and Research Bhopal, Bhopal, Madhya Pradesh, INDIA.}
\affiliation{\sup{4} Biomedical Engineering Department, Oregon Health \& Science University, Portland, OR, USA.}
\affiliation{\sup{5} Indian Statistical Institute, Kolkata, West Bengal-700108, INDIA}
\email{shiva.jolad@iitgn.ac.in, mido@vt.edu, 	shastry@ohsu.edu, gadgilm@iiserb.ac.in, ayanbasu@isical.ac.in }
}


\abstract{We introduce a new one-parameter family of divergence measures, called bounded Bhattacharyya distance (BBD) measures, for quantifying the dissimilarity between probability distributions. These measures are bounded, symmetric and positive semi-definite and do not require absolute continuity. In the asymptotic limit, BBD measure approaches the squared Hellinger distance. A generalized BBD measure for multiple distributions is also introduced. We prove an extension of a theorem of Bradt and Karlin for BBD relating Bayes error probability and divergence ranking. We show that BBD belongs to the class of generalized Csiszar f-divergence and derive some properties such as curvature and relation to Fisher Information. For distributions with vector valued parameters, the curvature matrix is related to the Fisher-Rao metric. We derive certain inequalities between BBD and well known measures such as Hellinger and Jensen-Shannon divergence. We also derive bounds on the Bayesian error probability. We give an application of these measures to the problem of signal detection where we compare two monochromatic signals buried in white noise and differing in frequency and amplitude. }

\onecolumn \maketitle \normalsize \vfill

\section{\uppercase{Introduction}}
\label{sec:introduction}

\noindent Divergence measures for the distance between two probability distributions are a statistical approach to comparing data and have been extensively studied in the last six decades \cite{Kullback1951, ali1966,kapur1984,Kullback1968,KumarKapur1986}.   These measures are widely used in varied fields such as pattern recognition \cite{Basseville1989, Ben1978,Choi2003},  speech recognition \cite{qiao2010study,lee1991information},  signal detection \cite{Kailath1967,kadota1967,poor1994introduction}, Bayesian model validation \cite{tumer1996} and quantum information theory \cite{nielsen2000,Lamberti2008}. Distance measures try to achieve two main objectives (which are not mutually exclusive):  to assess (1) how ``close'' two distributions are compared to others and (2) how ``easy'' it is to distinguish between one pair than the other \cite{ali1966}. 

There is a plethora of distance measures available to assess the convergence (or divergence) of probability distributions. Many of these measures are not metrics in the strict mathematical sense, as they may not satisfy either the symmetry of arguments or the triangle inequality. In applications, the choice of the measure depends on the interpretation of the metric in terms of the problem considered,  its analytical properties and ease of computation \cite{Gibbs2002}. One of the most well-known and widely used divergence measures, the Kullback-Leibler divergence (KLD)\cite{Kullback1951,Kullback1968}, can create problems in specific applications. Specifically, it is unbounded above and requires that the distributions be \emph{absolutely continuous} with respect to each other. Various other information theoretic measures have been introduced keeping in view ease of computation ease and utility in problems of signal selection and pattern recognition.  Of these measures, Bhattacharyya distance \cite{Bhattacharyya1946, 
Kailath1967, Nielsen2011} and Chernoff distance \cite{Chernoff1952,Basseville1989, Nielsen2011} have been widely used in signal processing. However, these measures are again unbounded from above. Many bounded divergence measures such as Variational, Hellinger distance \cite{Basseville1989,Dasgupta2011} and Jensen-Shannon metric \cite{Burbea1982,Rao1982b,Lin1991} have been studied extensively. Utility of these measures vary depending on properties such as tightness of bounds on error probabilities, information theoretic interpretations, and the ability to generalize to multiple probability distributions. 

Here we introduce a new one-parameter ($\alpha$) family of bounded measures based on the Bhattacharyya coefficient, called bounded Bhattacharyya distance (BBD) measures. These measures are symmetric, positive-definite and bounded between 0 and 1. In the asymptotic limit ($\alpha \rightarrow \pm \infty$) they approach squared Hellinger divergence \cite{hellinger1909,kakutani1948}. Following Rao \cite{Rao1982b} and Lin \cite{Lin1991}, a generalized BBD is introduced to capture the divergence (or convergence) between multiple distributions. We show that BBD measures belong to the generalized class of f-divergences and inherit useful properties such as curvature and its relation to Fisher Information.  Bayesian inference is useful in problems where a decision has to be made on classifying an observation into one of the possible array of states, whose prior probabilities are known \cite{Hellman1970,varshney2008quantization}. Divergence measures are useful in estimating the error in such classification \cite{Ben1978,Kailath1967, varshney2011bayes}.  We prove an extension of the Bradt Karlin theorem for BBD, which proves the existence of prior probabilities relating Bayes error probabilities with ranking based on divergence measure. Bounds on the error probabilities $P_e$ can be calculated through BBD measures using certain inequalities  between Bhattacharyya coefficient and $P_e$. We derive two inequalities for a special case of BBD ($\alpha=2$) with Hellinger and Jensen-Shannon divergences. Our bounded measure with $\alpha=2$ has been used by Sunmola \cite{sunmola2013optimising} to calculate distance between Dirichlet distributions in the context of Markov decision process. We illustrate the applicability of BBD measures by focusing on signal detection problem that comes up in areas such as gravitational wave detection \cite{finn1992detection}. Here we consider discriminating two monochromatic signals, differing in frequency or amplitude, and corrupted with additive white noise.  We compare the Fisher Information of the BBD measures with that of KLD and Hellinger distance for these random processes, and highlight the regions where FI is insensitive large parameter deviations. We also characterize the performance of BBD for different signal to noise ratios, providing thresholds for signal separation.

Our paper is organized as follows: Section I is the current introduction. In Section II, we recall the well known Kullback-Leibler and Bhattacharyya divergence measures, and then introduce our bounded Bhattacharyya distance measures. We discuss some special cases of BBD, in particular Hellinger distance. We also introduce the generalized BBD for multiple distributions. In Section III, we show the positive semi-definiteness of BBD measure, applicability of the Bradt Karl theorem and prove that BBD belongs to generalized f-divergence class. We also derive the relation between curvature and Fisher Information,  discuss the curvature metric and prove some inequalities with other measures such as Hellinger and Jensen Shannon divergence for a special case of BBD.  In Section IV, we move on to discuss application to signal detection problem. Here we first briefly describe basic formulation of the problem,  and then move on computing distance between random processes and comparing BBD measure with Fisher Information and KLD.    In the Appendix we provide the expressions for BBD measures , with $\alpha=2$, for some commonly used distributions.  We conclude the paper with summary and outlook.

\section{\uppercase{Divergence measures}}
In the following subsection we consider a measurable space $\Omega$ with $\sigma$-algebra ${\cal B}$  and the set of all probability measures ${\cal M}$ on $(\Omega,{\cal B})$. Let $P$ and $Q$ denote probability measures on $(\Omega,{\cal B})$ with $p$ and $q$ denoting their densities with respect to a common measure $\lambda$. 
We recall the definition of absolute continuity \cite{Royden1986}:
\begin{definition}[Absolute Continuity]
 A measure $P$ on the Borel subsets of the real line is absolutely continuous with respect to Lebesgue measure $Q$, if $P(A)=0$, for every Borel subset $A\in {\cal B}$ for which $Q(A)=0$, and is denoted by $P<<Q$.       
\end{definition}

\subsection{Kullback-Leibler divergence}
The Kullback-Leibler  divergence (KLD) (or relative entropy) \cite{Kullback1951,Kullback1968} between two distributions $P,Q$  with densities $p$ and  $q$ is given by:
\begin{equation}
I(P,Q)\equiv \int p\log\left(\frac{p}{q}\right) d\lambda.  
\end{equation} 
The symmetrized version is given by 
\[
J(P,Q)\equiv(I(P,Q)+I(Q,P))/2
\]
\cite{Kailath1967}, $I(P,Q)\in [0,\infty]$. It diverges if $ \exists
\quad x_0: q(x_0)=0$ and $p(x_0)\ne 0 $. 
 \newline

%The Kullback-Leibler  divergence (KLD) (or relative entropy) \cite{Kullback1951,Kullback1968} between two distributions $P,Q$  with densities $p(x)$ and  $q(x)$ is given by:
%\begin{equation}
%I(P,Q)\equiv \int p(x)\log\left(\frac{p(x)}{q(x)}\right) dx.  
%\end{equation} 
%The symmetrized version is given by $J(P,Q)\equiv(I(P,Q)+I(Q,P))/2$
%\cite{Kailath1967}, $I(P,Q)\in [0,\infty]$. It diverges if $ \exists
%\quad x_0: q(x_0)=0$ and $p(x_0)\ne 0 $. 
% \newline

KLD is defined only when $P$
is absolutely continuous w.r.t. $Q$. This feature can be problematic in numerical computations when the measured distribution has zero values. 

\subsection{Bhattacharyya Distance}
Bhattacharyya distance is a widely used measure in signal selection
and pattern recognition  \cite{Kailath1967}. It is defined as:
\begin{equation}
B(P,Q)\equiv -\ln \left(\int \sqrt{pq} d\lambda\right)=-\ln (\rho), 
\end{equation} 
where the term in parenthesis $\rho(P,Q)\equiv \int \sqrt{pq}
d\lambda$  is called Bhattacharyya coefficient \cite{Bhattacharyya1943, Bhattacharyya1946} in
pattern recognition, affinity in theoretical statistics, and fidelity
in quantum information theory. Unlike in the case of
KLD, the Bhattacharyya distance avoids the requirement of absolute continuity. It is a special case of Chernoff distance 
\[                                                                                                                                                                                                                                                                                      
C_\alpha(P,Q)\equiv -\ln \left(\int p^\alpha(x) q^{1-\alpha}(x) dx\right),                                                                                                                                                                                                                                                                                     \]
with $\alpha=1/2$. For discrete probability distributions, $\rho\in
[0,1]$ is interpreted as a scalar product of the probability vectors
$\mathbf{P}=(\sqrt{p_1},\sqrt{p_2},\dots,\sqrt{p_n})$ and
$\mathbf{Q}=(\sqrt{q_1},\sqrt{q_2},\dots,\sqrt{q_n})$. Bhattacharyya
distance is symmetric, positive-semidefinite, and unbounded ($0\le B\le \infty$).  It is finite as long as there exists some region $S\subset X$ such that whenever $x\in S :p(x)q(x)\ne 0$. 

\subsection{Bounded Bhattacharyya Distance Measures }
In many applications, in addition to the desirable properties of the
Bhattacharyya distance, boundedness is required.  We propose a new family of bounded measure of Bhattacharyya distance as below,
\begin{equation}
B_{\psi,b}(P,Q)\equiv -\log_b(\psi(\rho))
\end{equation} 
where, $\rho=\rho(P,Q)$ is the Bhattacharyya coefficient, $\psi_b(\rho)$ satisfies $\psi(0)=b^{-1}$ , $\psi(1)=1$.
In particular we choose the following form :
 \begin{eqnarray}
  \psi(\rho)&=&\left[1-\frac{(1-\rho)}{\alpha}\right]^\alpha \nonumber \\
  b&=&\left(\frac{\alpha}{\alpha-1}\right)^\alpha,
 \end{eqnarray}
 where $\alpha \in [-\infty,0) \cup (1,\infty] $. This gives the measure 
   \begin{equation}
  B_\alpha(\rho(P,Q))\equiv-\log_{\left(1-\frac{1}{\alpha}\right)^{-\alpha}} \left[1-\frac{(1-\rho)}{\alpha}\right]^\alpha, 
  \label{BBDFormn}
    \end{equation}
which can be simplified to
   \begin{equation}
  B_\alpha(\rho)=\frac{\log \left[1-\frac{(1-\rho)}{\alpha}\right] }{\log \left[1-\frac{1}{\alpha}\right] } .
  \label{BBDDefn}
    \end{equation}
 It is easy to see that $B_\alpha(0)=1,~B_\alpha(1)=0$. 

\subsection{Special cases}  
   \begin{enumerate}
     \item For $\alpha=2$ we get,  
\begin{equation}
  	  B_{2}(\rho)=-\log_{2^2}\left[\frac{1+\rho}{2}\right]^2 =-\log_2\left(\frac{1+\rho}{2}\right) .
  	\end{equation}
We study some of its special properties in Sec.\ref{Ineq}.
  	\item $\alpha\rightarrow \infty $ 	
	\begin{equation}
 	  B_{\infty}(\rho)=-\log_{e}e^{-(1-\rho)}=1-\rho=H^2(\rho),
 	\end{equation}
where $H(\rho)$ is the Hellinger distance \cite{Basseville1989,Kailath1967,hellinger1909,kakutani1948}  
 \begin{equation}
  H(\rho)\equiv \sqrt{1-\rho(P,Q)}.
\label{Hellinger}
 \end{equation} 
 	\item $\alpha=-1$
  	\begin{equation}
  	 B_{-1}(\rho)=-\log_2 \left(\frac{1}{2-\rho}\right)  .	 
  	\end{equation} 
  	\item $\alpha\rightarrow -\infty$
  	\begin{equation}
  	B_{-\infty}(\rho)=  \log_{e} e^{(1-\rho)}=1-\rho=H^2(\rho) . 
  	\end{equation} 
  \end{enumerate}
We note that BBD measures approach squared Hellinger distance when $\alpha \rightarrow \pm \infty $. In general, they are convex (concave) when $\alpha>1$ ($\alpha<0$) in $\rho$, as seen by evaluating second derivative 
\begin{eqnarray}
\frac{\partial^2 B_\alpha (\rho) }{\partial \rho^2}  &=&\frac{-1}{\alpha^2 \log\left(1-\frac{1}{\alpha}\right) \left(1-\frac{1-\rho}{\alpha}\right)^2  }= 
\nonumber \\
&=&
\begin{cases}
   > 0 & ~\alpha >1 \\
   < 0 & ~\alpha< 0 ~~.                                                                                                   
 \end{cases} 
\end{eqnarray} 
From this we deduce $B_{\alpha>1}(\rho)\le H^2(\rho) \le B_{\alpha<0}(\rho) $ for $\rho\in[0,1]$. A comparison between Hellinger and BBD measures for different values of $\alpha$ are shown in Fig. \ref{fig:BBDCompare}.

\begin{figure}[!t]
\centering
\fbox{\includegraphics[scale=0.4]{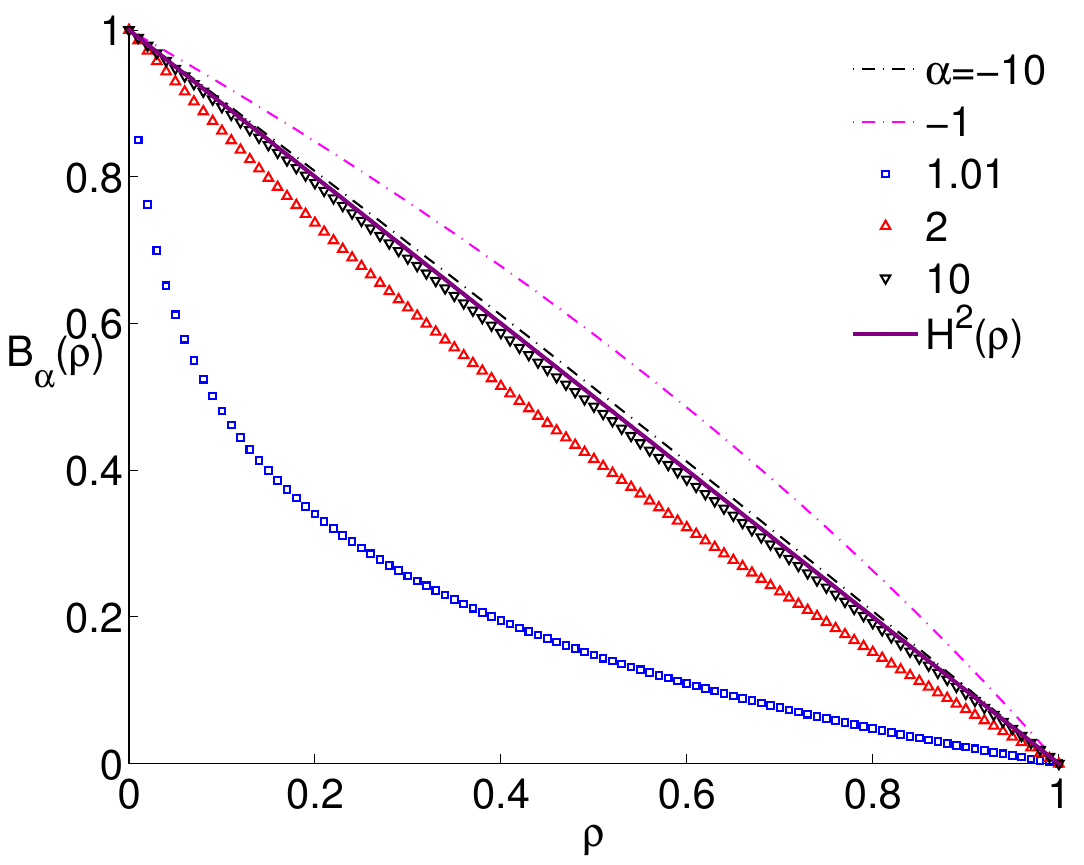}}
\caption{[Color Online] Comparison of Hellinger and  bounded Bhattacharyya distance measures for different values of $\alpha$.}
\label{fig:BBDCompare}
\end{figure}

\subsection{Generalized BBD measure}
In decision problems involving more than two random variables, it is very useful to have divergence measures involving more than two distributions \cite{Lin1991,Rao1982a,Rao1982b}.  
We use the generalized geometric mean ($G$) concept to define bounded Bhattacharyya measure for more than two distributions.   
The $G_{\boldsymbol\beta}(\{p_i\})$ of $n$ variables $p_1,p_2,\dots, p_n$ with weights $\beta_1,\beta_2,\dots, \beta_n$, such that $\beta_i\ge 0, ~\sum_i \beta_i=1$, is given by
\[
G_{\boldsymbol\beta}(\{p_i\})= \prod_{i=1}^{n} p_{i}^{\beta_{i}}. 
\]
For $n$ probability distributions $P_1,P_2,\dots,P_n~$, with densities $p_1,p_2,\dots,p_n~$, we define a generalized Bhattacharyya coefficient, also called Matusita measure of affinity \cite{matusita1967notion,toussaint1974some}: 
\begin{equation} 
\rho_{\boldsymbol\beta}(P_1,P_2,\dots , P_n)=\int_\Omega \prod_{i=1}^{n} p_{i}^{\beta_{i}} d\lambda .
\end{equation}
where $\beta_i\ge 0, ~\sum_i \beta_i=1$. Based on this, we define the generalized bounded Bhattacharyya measures as:
\begin{equation}
B_{\alpha}^{\boldsymbol \beta}(\rho_{\boldsymbol\beta}(P_1,P_2,\dots, P_n)) \equiv \frac{\log(1-\frac{1-\rho_{\boldsymbol\beta}}{\alpha})}{\log(1-1/\alpha)}
\label{GenBBD}
\end{equation}
where $\alpha \in [-\infty,0) \cup (1,\infty] $. For brevity we denote it as $B_{\alpha}^{\boldsymbol \beta}(\rho)$.
Note that, $0 \le \rho_{\boldsymbol\beta}\le 1$ and $ 0\leq  B_{\alpha}^{\boldsymbol \beta}(\rho) \leq 1 $, since the weighted geometric mean is maximized when all the $p_{i}$'s are the same, and minimized when any two of the probability densities $p_{i}$'s are perpendicular to each other. 

\section{\uppercase{Properties}}
\subsection{Symmetry, Boundedness and Positive Semi-definiteness}
\begin{theorem}
  $B_\alpha( \rho(P,Q))$  is symmetric, positive semi-definite and bounded in the interval $[0,1]$ for $\alpha \in [-\infty,0) \cup (1,\infty] $. 
\label{PosDefinite}
\end{theorem}

\begin{proof}[Proof]
Symmetry: Since   $\rho(P,Q)=\rho(Q,P)$, it follows that 
\[ 
B_\alpha(\rho(P,Q))=B_\alpha(\rho(Q,P)). 
\]
Positive-semidefinite and boundedness: Since $B_\alpha(0)=1$, $B_\alpha(1)=0$ and 
\[ 
\frac{\partial B_\alpha(\rho)}{\partial \rho}=\frac{1}{\alpha \log (1-1/\alpha) ~ [1-(1-\rho)/\alpha]}<0  
\]
for $0\le \rho \le 1$ and $\alpha \in [-\infty,0) \cup (1,\infty] $, it follows that
 \begin{equation}
       0\leq   B_\alpha(\rho) \leq 1.
      \end{equation} 
\end{proof}
      
%We also note, \begin{equation}
%B_\alpha (P,Q)=
%\begin{cases}
% 0 &  P=Q  ~~ \mathrm{ almost} ~ \mathrm{everywhere} \\
%1&  p(x)q(x)=0~~ \forall x\in X \\
%  (0,1)&    ~~~\mathrm{otherwise}  
%\end{cases} .
%\end{equation}

\subsection{Error Probability and Divergence Ranking}
Here we recap the definition of error probability and prove the applicability of  Bradt and Karlin \cite{bradt1956} theorem to BBD measure.
\begin{definition}[Error probability:]
The optimal Bayes error probabilities (see eg: \cite{Ben1978, Hellman1970, toussaint1978probability}) for classifying two events $P_1,P_2$ with densities $p_1(x)$ and $p_2(x)$ with prior probabilities $\Gamma=\{\pi_1,\pi_2\}$ is given by 
\begin{equation}
 P_e=\int \min[\pi_1p_1(x),\pi_2 p_2(x)]dx.
\end{equation} 
\end{definition}

{\bf Error comparison}: Let $p_i^\beta(x)$ ($i=1,2$) be parameterized by $\beta$ (Eg: in case of Normal distribution $\beta=\{\mu_1, \sigma_1; \mu_2, \sigma_2 \}$ ). In signal detection literature,  a signal set $\beta$ is considered better than set $\beta'$ for the densities $p_i(x)$ , when the error probability is less for  $\beta$ than for $\beta'$ (i.e. $P_e(\beta) < P_e(\beta')$) \cite{Kailath1967}. \\

{\bf Divergence ranking}: We can also rank the parameters by means of some  divergence $D$. The signal set $\beta$ is better (in the divergence sense) than $\beta'$, if $D_\beta(P_1,P_2)> D_{\beta'}(P_1,P_2)$. \\ 

In general it is {\it not} true that  $D_\beta(P_1,P_2)>D_{\beta'}(P_1,P_2) \implies P_e({\beta})<P_e(\beta')$. Bradt and Karlin  proved the following theorem relating error probabilities and divergence ranking for \emph{symmetric} Kullback Leibler divergence $\mathbf{J}$:
\begin{theorem}[Bradt and Karlin \cite{bradt1956}]
\label{BradtKarlin}
 If  $J_\beta(P_1,P_2)>J_{\beta'}(P_1,P_2)$, then $\exists$ a set of prior probabilities $\Gamma=\{\pi_1, \pi_2\}$ for two hypothesis
$g_1,g_2$, for which
\begin{equation}
P_e(\beta,\Gamma)<P_e(\beta',\Gamma) 
\end{equation} 
where $P_e(\beta,\Gamma)$ is the error probability with parameter $\beta$ and prior probability $\Gamma$.
\end{theorem}
It is clear that the theorem asserts existence, but no method of finding these prior probabilities. Kailath \cite{Kailath1967} proved the applicability of the Bradt Karlin Theorem for Bhattacharyya distance measure. 
We follow the same route and show that the $B_{\alpha}(\rho)$ measure satisfies a similar property using the following theorem by Blackwell.
\begin{theorem}[Blackwell \cite{blackwell1951}]
\label{Blackwell}
$P_{e}({\beta'},\Gamma) \leq  P_{e}(\beta,\Gamma) $ for all  prior probabilities $\Gamma$ if and only if 
\[
 \mathbb{E}_{{\beta'} }[ \Phi (L_{{\beta'} })|g] \leq   \mathbb{E}_{\beta}[ \Phi (L_{\beta })|g] , 
\]
$\forall$  continuous concave functions $\Phi(L)$, where $L_\beta=p_1(x,\beta)/p_2(x,\beta)$ is the likelihood ratio and $ \mathbb{E}_{\omega} [ \Phi (L_{\omega})|g]$ is the expectation of $\Phi(L_{\omega})$ under the hypothesis $g=P_2$.
\end{theorem}
\begin{theorem}
\label{KBExtend}
 If $B_\alpha(\rho ({\beta}))>B_\alpha(\rho ({\beta'}))$, or equivalently $\rho(\beta)<\rho({\beta'})$ then $\exists$ a set of prior probabilities $\Gamma=\{\pi_1, \pi_2\}$ for two hypothesis
$g_1,g_2$, for which
\begin{equation}
P_e(\beta,\Gamma)<P_e(\beta',\Gamma). 
\end{equation} 
%where $P_e(\beta,\pi)$ is the error probability with parameter $\beta$ and prior probability $\Gamma$.
\end{theorem}
\begin{proof} 
The proof closely follows Kailath \cite{Kailath1967}. First note that $\sqrt{L}$ is a concave function of $L$ (likelihood ratio) , and 
\begin{eqnarray}
\rho(\beta)&=&\sum_{x\in X} \sqrt{p_1(x,\beta)p_2(x,\beta)} \nonumber \\
 &=&\sum_{x\in X} \sqrt{\frac{p_1(x,\beta)}{p_2(x,\beta)}}p_2(x,\beta) \nonumber \\
  &=&\mathbb{E}_{\beta}\left[\sqrt{L_{\beta}}|g_2\right]. 
\end{eqnarray}
Similarly
\begin{equation}
\rho(\beta')= \mathbb{E}_{\beta'}\left[\sqrt{L_{\beta'}}|g_2\right]
\end{equation}
Hence, $\rho(\beta)<\rho(\beta')\Rightarrow $
\begin{equation}
  \mathbb{E}_{\beta}\left[\sqrt{L_{\beta}}|g_2\right]<\mathbb{E}_{\beta'}\left[\sqrt{L_{\beta'}}|g_2\right].
\label{Ealpha} 
\end{equation} 
Suppose assertion of the stated theorem is not true, then for all $\Gamma$, $P_e(\beta',\Gamma) \le P_e(\beta,\Gamma)$. Then by Theorem \ref{Blackwell}, $\mathbb{E}_{\beta' }[ \Phi (L_{\beta' })|g_2] \leq   \mathbb{E}_{\beta}[ \Phi (L_{\beta })|g_2] $ which contradicts our result in Eq. \ref{Ealpha}. 
%Hence, if $\rho(\beta)<\rho(\beta)$, then $P_{e}(\beta,\lambda)\geq P_{e}(\beta,\lambda)$
\end{proof}

\subsection{Bounds on Error Probability}
Error probabilities are hard to calculate in general. Tight bounds on $P_e$ are often extremely useful in practice. Kailath \cite{Kailath1967} has shown bounds on $P_e$ in terms of the Bhattacharyya coefficient $\rho$:
\begin{equation}
 \frac{1}{2}\left[2\pi_1-\sqrt{1-4\pi_1 \pi_2\rho^2}\right]\leq P_{e}\leq \left(\pi_1-\frac{1}{2}\right)+ \sqrt{\pi_1 \pi_2}\rho,
\end{equation} 
with $\pi_1+\pi_2=1$. If the priors are equal  $\pi_1=\pi_2=\frac{1}{2}$, the expression simplifies to 

\begin{equation}
 \frac{1}{2}\left[1-\sqrt{1-\rho^2}\right]\leq P_{e}\leq \frac{1}{2}\rho.
\end{equation} 
Inverting relation in Eq. \ref{BBDDefn} for $\rho(B_\alpha)$, we can get the bounds in terms of  $B_\alpha(\rho)$ measure. For the equal prior probabilities case, Bhattacharyya coefficient gives a tight upper bound for large systems when $\rho\rightarrow 0$ (zero overlap) and the observations are independent and identically distributed. These bounds are also useful to discriminate between two processes with arbitrarily low error probability \cite{Kailath1967}. We suppose that tighter upper bounds on error probability can be derived through Matusita's measure of affinity \cite{bhattacharya1982upper, toussaint1977upper, toussaint1975sharper}, but is beyond the scope of present work.

\subsection{f-divergence}
\label{fdiv}
A class of divergence measures called f-divergences were introduced by Csiszar \cite{Csiszar1967,Csiszar1975}  and independently by Ali and Silvey \cite{ali1966} (see \cite{Basseville1989} for review). It encompasses many well known divergence measures including KLD, variational, Bhattacharyya and Hellinger distance. In this section, we show that $B_\alpha(\rho)$ measure for $\alpha\in (1,\infty]$, belongs to the generic class of f-divergences defined by Basseville \cite{Basseville1989}.  
\begin{definition}[f-divergence \cite{Basseville1989}]
Consider a measurable space $\Omega$ with $\sigma$-algebra ${\cal B}$.  Let $\lambda$ be a measure on $(\Omega,{\cal B})$  such that any probability laws $P$ and $Q$ are absolutely continuous with respect to $\lambda$, with densities $p$ and $q$. Let $f$ be a continuous {\it convex} real function on $\mathbb{R}^+$, and  $g$ be an {\it increasing} function on $\mathbb{R}$. The class of divergence coefficients between two probabilities:
\begin{equation}
 d(P,Q)=g\left( \int_\Omega  f \left(\frac{p}{q}\right)q d\lambda \right)
\end{equation} 
are called the f-divergence measure w.r.t.  functions $(f,g)$ . Here $p/q=L$ is the likelihood ratio.  The term in the parenthesis of $g$ gives the  Csiszar's \cite{Csiszar1967,Csiszar1975} definition of f-divergence.  
\end{definition}
The $B_\alpha(\rho(P,Q))$ , for $\alpha\in (1,\infty]$ measure can be written as the following $f$ divergence:
\begin{equation}
 f(x)=-1+\frac{1-\sqrt{x}}{\alpha}, ~~ g(F)=\frac{\log(-F)}{\log(1-1/\alpha)}, 
\end{equation} 
where,
\begin{eqnarray}
   F&=&\int_\Omega  \left[ -1+\frac{1}{\alpha} \left(1-\sqrt{\frac{p}{q}}\right) \right]  q  d\lambda \nonumber \\
    &=& \int_\Omega\left[ q\left(-1+\frac{1}{\alpha}\right)-\frac{1}{\alpha} \sqrt{pq} \right]  d\lambda \nonumber \\
    &=&-1+\frac{1-\rho}{\alpha} .
\end{eqnarray} 
and
\begin{equation}
    g(F) =\frac{\log(1-\frac{1-\rho}{\alpha})}{\log(1-1/\alpha)}=B_\alpha(\rho(P,Q)).
\end{equation} 

\subsection{Curvature and Fisher Information}
\label{CurvFI}
In statistics,  the information that an observable random variable $X$
carries about an unknown parameter $\theta$ (on which it depends) is
given by the Fisher information. One of the important properties of
f-divergence of two distributions of the same parametric family is
that their curvature measures the Fisher information. Following the
approach pioneered by Rao \cite{Rao1945}, we relate
the curvature of BBD measures to the Fisher information and derive the differential curvature metric. The discussions below closely follow  \cite{Dasgupta2011}.

\begin{definition}
 Let \{$f(x|\theta); \theta \in \Theta \subseteq\mathbb{R} $\}, be a family of densities indexed by real parameter $\theta$, with some regularity conditions ($f(x|\theta)$ is absolutely continuous).  
\begin{equation}
 Z_\theta(\phi)\equiv B_\alpha(\theta,\phi)=\frac{\log(1-\frac{1-\rho(\theta,\phi)}{\alpha})}{\log(1-1/\alpha)}
\end{equation}
where $\rho(\theta,\phi)=\int \sqrt{f(x|\theta)f(x|\phi)}dx$
\end{definition}

\begin{theorem}
\label{Theorem:FisherInfo}
 Curvature of $Z_\theta(\phi)|_{\phi=\theta}$ is the Fisher information of $f(x|\theta)$ up to a multiplicative constant.
\end{theorem}

\begin{proof}
 Expand $Z_\theta(\phi)$ around theta
 \begin{equation}
 Z_\theta(\phi)= Z_\theta(\theta)+(\phi-\theta)\frac{dZ_\theta(\phi)}{d\phi}  +\frac{(\phi-\theta)^2}{2}\frac{d^2Z_\theta(\phi)}{d\phi^2} +\dots   
\label{ZTP}
 \end{equation}
Let us observe some properties of Bhattacharyya coefficient  : 
$ \rho(\theta,\phi)= \rho(\phi,\theta),~ ~\rho(\theta,\theta)= 1$, and its derivatives:
\begin{equation}
\hspace{-0.6cm} 
\frac{\partial \rho(\theta,\phi)}{\partial \phi}  \Big |_{\phi=\theta} =  \frac{1}{2}\frac{\partial }{\partial \theta} \int f(x|\theta)dx=0, 
\end{equation}

\begin{align}
 \frac{\partial^2 \rho(\theta,\phi)}{\partial \phi^2}\Big |_{\phi=\theta}=&-\frac{1}{4}\int \frac{1}{f}  \left (\frac{\partial f}{\partial \theta}\right)^2 dx  +\frac{1}{2} \frac{\partial^2 }{\partial \theta^2}\int f dx \nonumber \\
=&-\frac{1}{4}\int f(x|\theta) \left (\frac{\partial \log f(x|\theta)}{\partial \theta}\right)^2 dx\nonumber \\
=&-\frac{1}{4} I_f(\theta).
\end{align}
where $I_f(\theta)$ is the Fisher Information of distribution $f(x|\theta)$
\begin{equation}
 I_f(\theta)=\int f(x|\theta) \left (\frac{\partial \log f(x|\theta)}{\partial \theta}\right)^2 dx.
 \label{FisherInfo}
\end{equation}
Using the above relationships, we can write down the terms in the expansion of Eq. \ref{ZTP}  $Z_\theta(\theta)= 0 ~,~ \frac{\partial Z_\theta(\phi)}{\partial \phi}\Big|_{\phi=\theta} =0 $, and
\begin{equation}
\frac{{\partial^2 Z_\theta(\phi)}}{{\partial \phi^2}}\Big |_{\phi=\theta}=C(\alpha)I_f(\theta) >0,
\end{equation}
where $C(\alpha)=\frac{-1}{4\alpha\log(1-1/\alpha)}>0$ 
\end{proof}
%\begin{eqnarray}
%Z_\theta(\theta)&=& 1 \nonumber \\
%\frac{\partial Z_\theta(\phi)}{\partial \phi}\Big|_{\phi=\theta}&=&0 \non \\
%\frac{\partial^2 Z_\theta(\phi)}{\partial \phi^2}\Big |_{\phi=\theta}&=&C(\alpha)I_f(\theta) >0
%\end{eqnarray}
\hspace{-0.4cm} The leading term of $B_\alpha(\theta,\phi)$ is given by  
 \begin{equation}
 B_\alpha(\theta,\phi)  \sim \frac{(\phi-\theta)^2}{2}C(\alpha)I_f(\theta)  . 
\label{Calpha}
 \end{equation}

\subsection{Differential Metrics }
Rao \cite{Rao1987} generalized the Fisher information to multivariate densities with vector valued parameters to obtain a ``geodesic'' distance between two parametric distributions $P_\theta, P_\phi$ of the same family. The Fisher-Rao metric has found applications in many areas such as image structure and shape analysis \cite{Maybank2004,Peter2006} , quantum statistical inference \cite{Brody1998} and Blackhole thermodynamics \cite{Quevedo2008}.
We derive such a metric for BBD measure using property of f-divergence.

Let $\theta,\phi \in \Theta \subseteq  \mathbb{R}^p$, then using the fact that
$ \frac{\partial Z(\theta,\phi)}{\partial \theta_i}\Big |_{\phi=\theta}=0,$
we can easily show that
\begin{equation}
 dZ_{\theta}=\sum_{i,j=1}^{p}\frac{\partial^2 Z_\theta}{\partial \theta_i \partial \theta_j}d\theta_i d \theta_j+\dots =\sum_{i,j=1}^{p}g_{ij}d\theta_id\theta_j+\dots.
\end{equation}
The curvature metric $g_{ij}$ can be used to find the geodesic on the curve $\eta(t),~t \in [0,1]$ with
\begin{equation}
\label{curveboundary}
{ \cal C}=\eta(t): ~~~\eta(0)=\theta ~~ \eta(1)=\phi. 
\end{equation} 
Details of the geodesic equation are given in many standard  differential geometry books. In the context of probability  distance measures reader is referred to (see 15.4.2 in A DasGupta \cite{Dasgupta2011} for details)
The curvature metric of all Csiszar f-divergences are just scalar multiple KLD measure 
\cite{Dasgupta2011,Basseville1989} given by:
\begin{equation}
g^{f}_{ij}(\theta)= f''(1)g_{ij}(\theta).
\end{equation}  
For our BBD measure 
\begin{eqnarray}
 f''(x)&=&\left(-1+\frac{1-\sqrt{x}}{\alpha}\right)''=\frac{1}{4\alpha x^{3/2}} \nonumber \\
	  \tilde{f}''(1)&=&1/4\alpha.
\end{eqnarray}
Apart from the $-1/\log(1-\frac{1}{\alpha})$, this is same as $C(\alpha)$ in Eq. \ref{Calpha}. It follows that the geodesic distance for our metric is same  KLD geodesic distance up to a multiplicative factor. KLD geodesic distances are tabulated in DasGupta \cite{Dasgupta2011}.

\subsection{Relation to other measures }
\label{Ineq}
Here we focus on the special case $\alpha=2$, i.e. $B_{2}(\rho)$  
%\begin{equation}
% \zeta(P,Q) = B_{2}(\rho(P,Q)) 
%\end{equation}  

% \begin{theorem}
%  $0\le \zeta(P,Q) \le H(P,Q)$. 
% \label{HZetaXiIEq1}
% \end{theorem}
% 
% \begin{proof}
% We have already shown that $0 \le B_\alpha(\rho)\le 1$ in Theorem \ref{PosDefinite}.
% We use the generalized Bernoulli inequality
% \begin{equation}
% (1+x)^r\le 1+rx , \quad 0\le r\le 1, \quad x>-1.
% \end{equation}
% 
% Set $x=1$ and $r=\rho=\int \sqrt{p(x)q(x)} dx$, the Bhattacharyya coefficient. 
% Hence
% \begin{eqnarray}
% 0 &&\le (1+1)^\rho =2^\rho\le 1+\rho\le 2 \nonumber\\ 
% %0 &&\le \log_2 2^\rho \le \log_2( 1+\rho) \le 1 \nonumber \\ 
% 0 &&\le \rho \le \log_2 (1+\rho)  \le 1  \nonumber \\ 
% %1 &&\ge 1-\rho \ge 1-\log_2 (1+\rho)\ge 0 \nonumber \\
% 1 &&\ge \sqrt{1-\rho} \ge \sqrt{1-\log_2 (1+\rho)}\ge 0  \nonumber \\
% 1 &&\ge \sqrt{1-\rho} \ge \sqrt{-\log_2 \frac{(1+\rho)}{2}}\ge 0. 
% \end{eqnarray}
% % \begin{align}
% % 0 &\le (1+1)^\rho =2^\rho\le 1+\rho\le 2 \nonumber\\ 
% % 0 &\le \log_2 2^\rho \le \log_2( 1+\rho) \le 1 \nonumber \\ 
% % 0 &\le \rho \le \log_2 (1+\rho) \le 1  \nonumber \\ 
% % 1 &\ge 1-\rho \ge 1-\log_2 (1+\rho)\ge 0 \nonumber \\
% % 1 &\ge \sqrt{1-\rho} \ge \sqrt{1-\log_2 (1+\rho)}\ge 0  \nonumber \\
% % 1 &\ge \sqrt{1-\rho} \ge \sqrt{-\log_2 \frac{(1+\rho)}{2}}\ge 0 \nonumber.
% % \end{align}
% Hence we get
% \begin{equation}
% 1\ge H\ge \xi\ge 0.
% \end{equation}
% \end{proof}

\begin{theorem}
\label{HXiIEq2}
\begin{equation}
B_2 \leq H^2 \leq \log {4} ~ B_2 
\end{equation} 
where 1 and  $\log{4}$ are sharp.
\end{theorem}

\begin{proof}
Sharpest  upper bound is achieved via taking 
$\mathrm{sup}_{\rho \in[0,1)} \frac{H^2(\rho)}{B_2(\rho)}$. Define 
\begin{eqnarray}
  g(\rho) &\equiv & \frac{1-\rho}{-\log_2{(1+\rho)/2}}.
\end{eqnarray}
We note that $g(\rho)$ is continuous and has no singularities whenever $\rho\in[0,1)$. 
Hence
\[
 g'(\rho)=\frac{\frac{1-\rho}{1+\rho}+\log(\frac{1+\rho}{2})}{\log^2{\frac{\rho+1}{2}}} \log{2}\geq 0 .
\] 
It follows that $g(\rho)$ is non-decreasing and hence $\mathrm{sup}_{\rho \in[0,1)} g(\rho)=\lim_{\rho \to 1}g(\rho)=\log(4)$.   Thus 
\begin{equation}
 H^2/B_2 \leq \log {4} .
\end{equation} 
Combining this with convexity property of $B_\alpha(\rho)$ for $\alpha >1$, we get 
\[
B_2 \leq H^2 \leq \log{4} ~ B_2  
\]
Using the same procedure we can prove a generic version of this inequality for $\alpha \in (1,\infty]$ , given by
\begin{equation}
 B_\alpha(\rho) \leq H^2 \leq -\alpha \log \left(1-\frac{1}{\alpha}\right) ~ B_\alpha(\rho)  
\end{equation} 

\end{proof}

\begin{definition}[Jensen-Shannon Divergence: ]
The Jensen difference between two distributions $P, Q$, with densities $p,q$ and weights
$(\lambda_1,\lambda_2 ); ~\lambda_1+\lambda_2=1$, is defined as,
\begin{equation}
 {\cal J}_{\lambda_1,\lambda_2}(P,Q)=H_s(\lambda_1 p+\lambda_2 q)-\lambda_1 H_s(p)-\lambda_2 H_s(q),
\end{equation} 
where $H_s$ is the Shannon entropy. Jensen-Shannon divergence (JSD) \cite{Burbea1982,Rao1982b,Lin1991} is based on the Jensen difference and is given by:
\begin{align}
JS(P,Q)=&{\cal J}_{1/2,1/2}(P,Q) \nonumber \\
=&\frac{1}{2} \int \Big[ p\log\left(\frac{2p}{p+q}\right)  \nonumber \\
& + q\log\left(\frac{2q}{p+q}\right) \Big] d\lambda 
\end{align}
\end{definition}
The structure and goals of JSD and BBD measures are similar. The following theorem compares the two metrics using Jensen's inequality. 
\begin{lemma}
\textbf{Jensen's Inequality:} For a convex function $\psi$, $\mathbb{E}[\psi(X)]\geq\psi(\mathbb{E}[X])$. 
\end{lemma}

\begin{theorem}[Relation to Jensen-Shannon measure]
$ JS(P,Q)\geq \frac{2}{\log{2}}B_2(P,Q) - \log{2}$
\end{theorem}
We use the un-symmetrized Jensen-Shannon metric for the proof.
\begin{proof}
\begin{align*}
\hspace{-0.7cm} JS(P,Q) =& \int{p(x)\log{\frac{2p(x)}{p(x)+q(x)}} dx}  \\
%&=& -\int{p(x)\log{\frac{p(x)+q(x)}{2p(x)}} dx} \nonumber  \\
=& -2\int{p(x)\log{\frac{\sqrt{p(x)+q(x)}}{\sqrt{2p(x)}}}dx}  \\
  \geq& -2\int{p(x)\log{\frac{\sqrt{p(x)}+\sqrt{q(x)}}
 {\sqrt{2p(x)}}}} dx \\ 
 & (\textmd{since}~ \sqrt{p+q}\leq\sqrt{p}+\sqrt{q}) \\ 
=&  \mathbb{E}_P\left[-2\log{\frac{\sqrt{p(X)}+\sqrt{q(X)}}{\sqrt{2p(X)}}}\right]
\end{align*}
By Jensen's inequality \\
  $~\mathbb{E}[-\log f(X)]\ge -\log \mathbb{E}[f(X)]$, we have 
  \begin{align*}
 \mathbb{E}_P  &\left[-2\log{\frac{\sqrt{p(X)}+ \sqrt{q(X)}}{\sqrt{2p(X)}}}\right] \geq  \\
  &-2\log{\mathbb{E}_P\left[\frac{\sqrt{p(X)}+\sqrt{q(X)}}{\sqrt{2p(X)}}\right]} .  
  \end{align*}
Hence,
\begin{eqnarray}
JS(P,Q) &\ge & -2\log{\int{p(x)\frac{\left(\sqrt{p(x)}+\sqrt{q(x)}\right)}{\sqrt{2p(x)}} dx}} \nonumber \\
 &=& -2\log{\left(\frac{1+\int{\sqrt{p(x)q(x)}}}{2}\right)}-\log{2}\nonumber \\
 &=& 2\left(\frac{B_2(p(x),q(x))}{\log{2}}\right)-\log{2}\nonumber \\
 &=& \frac{2}{\log{2}}B_2(P,Q)-\log{2}.
\end{eqnarray}
\end{proof}

%\begin{figure}[!h]
%  %\vspace{-0.2cm}
%  \centering
%   {\epsfig{file = SCITEPRESS.eps, width = 5.5cm}}
%  \caption{This caption has one line so it is centered.}
%  \label{fig:example1}
% \end{figure}
%
%\begin{figure}[!h]
%  \vspace{-0.2cm}
%  \centering
%   {\epsfig{file = SCITEPRESS.eps, width = 5.5cm}}
%  \caption{This caption has more than one line so it has to be justified.}
%  \label{fig:example2}
%  \vspace{-0.1cm}
%\end{figure}

\section{\uppercase{Application to signal detection}}

Signal detection is a common problem occurring in many fields such as communication engineering, pattern recognition, and Gravitational wave detection \cite{poor1994introduction}. In this section, we briefly describe the problem and terminology used in signal detection.  We illustrate though simple cases how divergence measures, in particular BBD can be used for discriminating and detecting signals buried in white noise of correlator receivers (matched filter).  For greater details  of the formalism used we refer the reader to review articles in the context of Gravitational wave detection by Jaranowski and Kr\'olak\cite{jaranowski2007gravitational} and  Sam Finn \cite{finn1992detection}.

One of the central problem in signal detection is to detect whether a deterministic signal $s(t)$ is embedded in an observed data $x(t)$, corrupted by noise $n(t)$. This can be posed as a hypothesis testing problem where the null hypothesis is absence of signal and alternative is its presence. 
 We take the noise to be additive , so that $x(t)=n(t)+s(t)$. We define the following terms used in signal detection: \textit{Correlation} $G$ (also called matched filter)  between $x$ and $s$, and \textit{signal to noise ratio} $\varrho$ \cite{finn1992detection,budzynski2008applications}.
\begin{equation}
 G = (x|s) ,~~~ \varrho = \sqrt{(s,s)} ,
 \label{matchedfiliter}
\end{equation}
where the scalar product $(.|.)$ is defined by
\begin{equation}
(x|y):=4\Re\int_0^\infty\!\frac{\tilde{x}(f)\tilde{y}^{*}(f)}{\tilde{N}(f)}\,\mathrm{d}f.
\label{scalarproduct}
\end{equation}
$\Re$ denotes the real part of a complex expression, tilde denotes the Fourier transform and the asterisk * denotes complex conjugation. $\tilde{N}$ is the \emph{one-sided spectral density of the noise}.\\
For white noise, the probability densities of $G$ when respectively signal is present and absent are given by \cite{budzynski2008applications}
\begin{equation}
p_{1}(G)= \frac{1}{\sqrt{\,2\pi}\varrho}\exp\left(\,-\frac{(G-\varrho^{2})^{2}}{2\varrho^{2}}\,\right),
\label{p_signal}
\end{equation}
\begin{equation}
p_{0}(G)=\frac{1}{\sqrt{\,2\pi}\varrho}\exp\left(\,-\frac{G^{2}}{2\varrho^{2}}\,\right)
\label{p_signal_ab}
\end{equation}

%The probabilities of false alarm and of detection are given by
%\begin{equation}
%P_{F}=\mathrm{erfc}\left(\frac{G_{0}}{\rho}\right),
%\end{equation}
%\begin{equation}
%P_{D}=\mathrm{erfc}\left(\frac{G_{0}}{\rho}-\rho\right).
%\end{equation}
%The error function is given by:
%\begin{equation}
%\mathrm{erfc}(x)=\frac{1}{\sqrt{2\pi}}\int_{x}^{\infty}\,\exp\left(-\frac{t^{2}}{2}\right)\,\mathrm{d}t,
%\end{equation}
%where $G_{0}$ is the threshold for the matched filter.\\
%We set an acceptable false alarm probability and find out the threshold $G_{0}$ and so the detection probability.

\subsection{Distance between Gaussian processes}
Consider a stationary Gaussian random process $\mathbf{X}$, which has signals $s_1$ or  $s_2$ with probability densities   $p_1$ and $p_2$ respectively of being present in it.  These densities follow the form Eq. \ref{p_signal} with signal to noise ratios  $\varrho_1^2$ and $\varrho_2^2$ respectively. The probability density  $p(\mathbf{X})$ of Gaussian process can modeled as limit of multivariate Gaussian distributions. The divergence measures between these processes $d(s_1,s_2)$  are  in general functions of the correlator $(s_1-s_2|s_1-s_2)$ \cite{budzynski2008applications}. Here we focus on distinguishing monochromatic signal $s(t)=A\cos(\omega t+\phi)$ and filter  $s_F(t)=A_F\cos(\omega_F t+\phi)$ (both buried in noise), separated in frequency or amplitude. 
  
 The Kullback-Leibler divergence between the signal and filter $I(s,s_F)$ is given by the correlation $(s-s_{F}|s-s_{F})$:
\begin{align}
I(s,s_F)=& (s-s_{F}|s-s_{F})= (s|s)+(s_F|s_F)-2(s|s_F) \nonumber \\ 
=&\varrho^{2}+\varrho_{F}^{2}-2\varrho\varrho_{F}[\langle\cos(\Delta\omega t)\rangle\cos(\Delta\phi)   \nonumber \\
&-\langle\sin(\Delta\omega t)\rangle\sin(\Delta\phi)],
\label{SF_correlate}
\end{align}
where $\langle\rangle$ is the average  over observation time $[0,T]$. Here we have assumed that noise spectral density $N(f)=N_0$ is constant over the frequencies $[\omega,\omega_F]$. 
The SNRs are given by
\begin{eqnarray}
\varrho^{2}=\frac{A^{2}T}{N_0} ,~~~ \varrho^{2}_{F}=\frac{A_{F}^{2}T}{N_0}.
\end{eqnarray}
(for detailed discussions we refer the reader to Budzynksi \emph{et. al} \cite{budzynski2008applications}). 

The Bhattacharyya distance between Gaussian processors with signals of same energy is  ( Eq 14 in \cite{Kailath1967}) just a multiple of the KLD $B=I/8$. We use this result to extract the  Bhattacharyya coefficient :
\begin{equation}
\rho(s,s_F)=\exp\left(-\frac{(s-s_{F}|s-s_{F})}{8}\right)
\end{equation}

\subsubsection{Frequency difference}
Let us consider the case when the SNRs of signal and filter are equal, phase difference is zero, but frequencies differ by $\Delta \omega$. 
The KL divergence is obtained by evaluating the correlator  in Eq. \ref{SF_correlate}  
\begin{equation}
I(\Delta \omega )=(s-s_{F}|s-s_{F})=2\varrho^{2}\left(1-\frac{\sin\left(\Delta\omega T\right)}{\Delta\omega T}\right).
\end{equation}
by noting $\langle\cos(\Delta\omega t)\rangle=\frac{\sin(\Delta\omega T)}{\Delta\omega T} ~~ \text{and}~~ \langle\sin(\Delta\omega t)\rangle=\frac{1-\cos(\Delta\omega T)}{\Delta\omega T} $.
Using this, the expression for BBD family can be written as 
\begin{equation}
B_{\alpha}(\Delta\omega)=\frac{\log\left( 1-\frac{1}{\alpha}\left[ 1-e^{-\frac{\varrho^{2}}{4}\left(1-\frac{\sin\left(\Delta\omega T\right)}{\Delta\omega T}\right)}\right]\right)}{\log\left(1-\frac{1}{\alpha}\right)}.
\end{equation}
As we have seen in section \ref{fdiv},  both BBD and KLD belong to the f-divergence family.  Their curvature for distributions belonging to same parametric family is a constant times the Fisher information (FI) (see Theorem: \ref{Theorem:FisherInfo}). Here we discuss where the BBD and KLD deviates from FI, when we account for higher terms in the expansion of these measures. 

The Fisher matrix element for frequency $g_{\omega, \omega}=E\left[\big(\frac{\partial \log \Lambda}{\partial \omega}\big)^2 \right]=\rho^2 T^2/3$ \cite{budzynski2008applications}, where $\Lambda$  is the likelihood ratio. Using the relation for line element $ds^2=\sum_{i,j}g_{ij}d \theta_i d \theta_j$ and noting that only frequency is varied, we get 
\begin{equation}
\mathrm{d}s=\frac{\varrho T\Delta\omega}{\sqrt{\,3}}.
\end{equation}
Using the relation between curvature of BBD measure and Fisher's Information in Eq.\ref{Calpha}, we can see that for low frequency differences the line element varies as:
\begin{equation}
\sqrt{\,\frac{2B_{\alpha}(\Delta\omega)}{C(\alpha)}}\sim\mathrm{d}s.	\nonumber
\end{equation}
Similarly $\sqrt{d_{KL}}\sim ds $ at low frequencies. However, at higher frequencies both KLD and BBD deviate from the Fisher information metric. In Fig. \ref{fig:FisherBBDCompareFrq}, we have plotted $\mathrm{d}s$, $\sqrt{\,d_{KL}}$ and $\sqrt{2B_{\alpha}(\Delta\omega)/C(\alpha)}$ with $\alpha=2$ and Hellinger distance ($\alpha\rightarrow
 \infty$)  for $\Delta\omega\in(0,0.1)$. We observe that till $\Delta \omega=0.01 $ (i.e. $\Delta \omega T \sim 1$), KLD and BBD follows Fisher Information and after that they start to deviate. This suggests that Fisher Information is not sensitive to large deviations.  There is not much difference between KLD, BBD and Hellinger for large frequencies due to the correlator $G$ becoming essentially a constant over a wide range of frequencies.  

\begin{figure}
\includegraphics[scale=0.4]{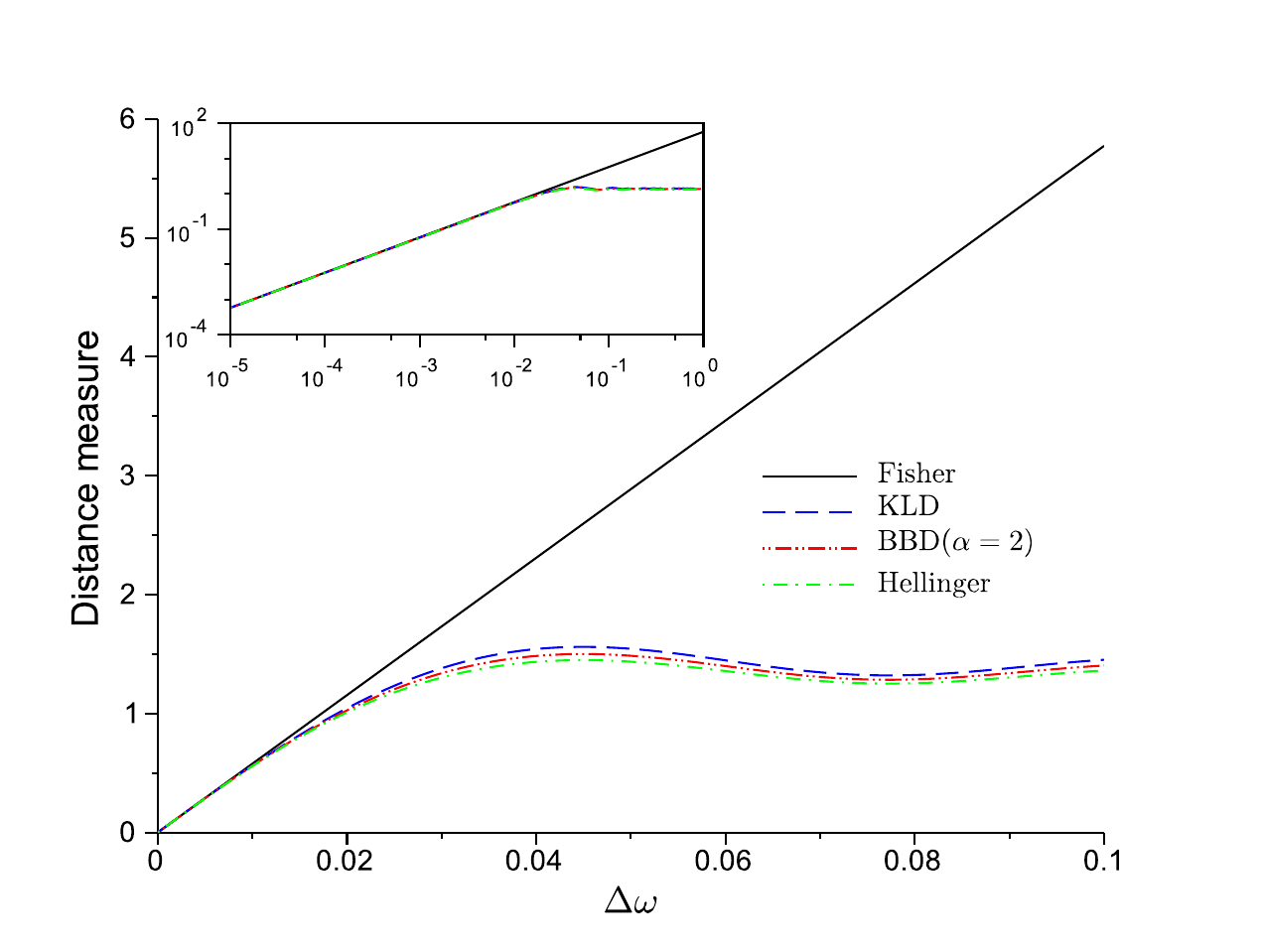}
\caption{Comparison of Fisher Information, KLD, BBD and Hellinger distance for two monochromatic signals differing by frequency $\Delta \omega$, buried in white noise. Inset shows wider range $\Delta\omega\in(0,1)$ .  We have set $\rho=1$ and chosen parameters $T=100$ and $N_0=10^4$.}
\label{fig:FisherBBDCompareFrq}
\end{figure}

\subsubsection{Amplitude difference}

\begin{figure}[!ht]
\includegraphics[scale=0.4]{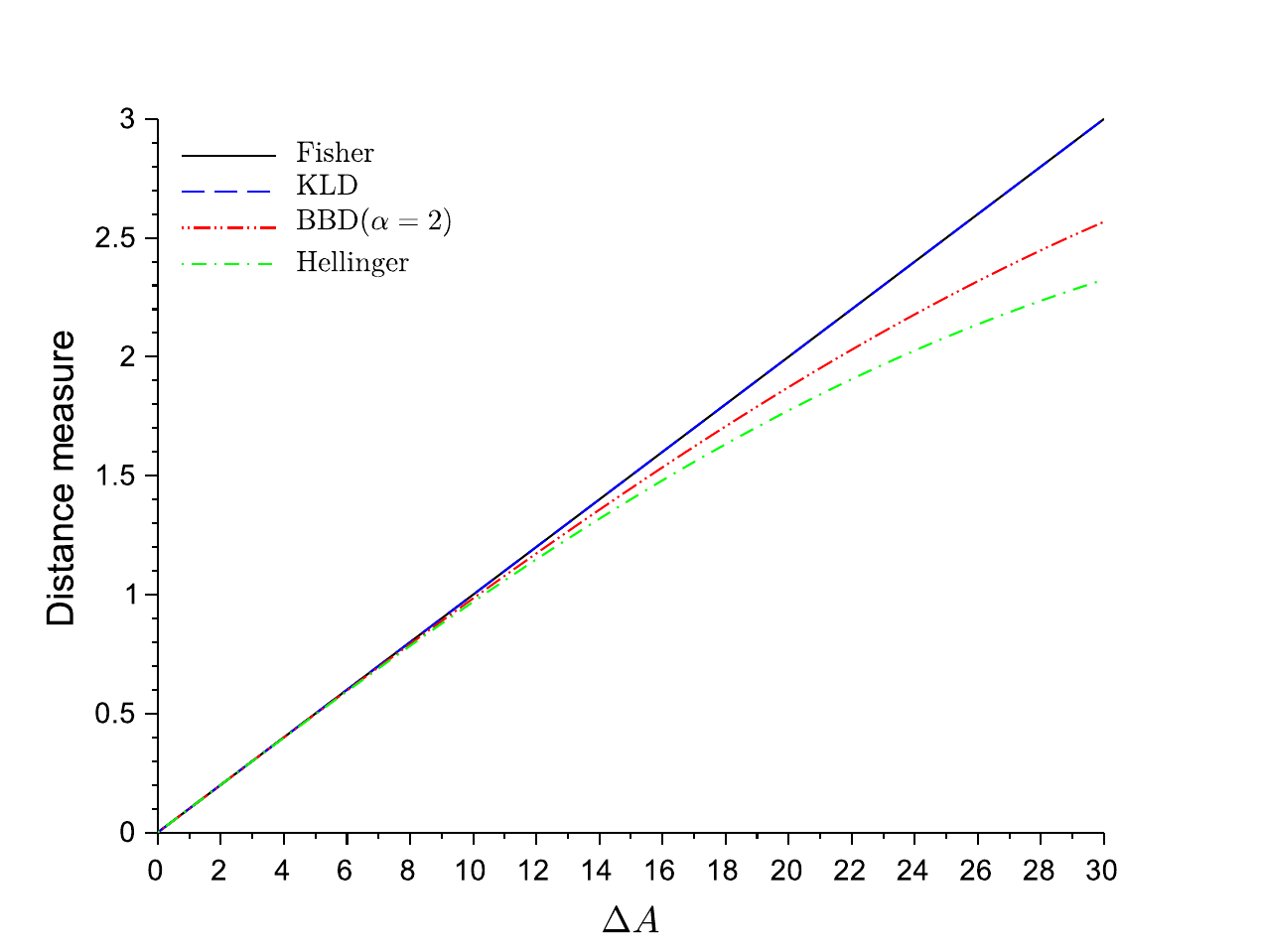}
\caption{Comparison of Fisher information line element with KLD, BBD and Hellinger distance for signals differing in amplitude and buried in white noise. We have set $A=1$,  $T=100$ and $N_0=10^4$.}
\label{fig:FisherBBDCompareAmp}
\end{figure}

We now consider the case where the frequency and phase of the signal and the filter are same but they differ in amplitude $\Delta A$ (which reflects in differing SNR). The correlation reduces to
\begin{align}
(s-s_{F}|s-s_{F})= &\frac{A^{2}T}{N_0}+\frac{(A+\Delta A)^{2}T}{N_0}-2\frac{A(A+\Delta A)T}{N_0} \nonumber \\
  =& \frac{(\Delta A)^{2}T}{N_0}.
\end{align}
This gives us $I(\Delta A)=\frac{(\Delta A)^{2}T}{N_0}$, which is the same as the line element $ds^2$ with Fisher metric $ds=\sqrt{T/2N_0}\Delta A$. 
In Fig. \ref{fig:FisherBBDCompareAmp}, we have plotted $\mathrm{d}s$, $\sqrt{\,d_{KL}}$ and $\sqrt{2B_{\alpha}(\Delta\omega)/C(\alpha)}$ for $\Delta A\in(0,40)$.
KLD and FI line element are the same. Deviations of BBD and Hellinger can be observed only for $\Delta A>10$.

Discriminating between two signals $s_1,s_2$ requires minimizing the error probability between them. By Theorem \ref{KBExtend}, there exists priors for which the problem translates into maximizing the divergence for BBD measures. For the monochromatic signals discussed above, the distance depends on parameters $(\rho_1,\rho_2,\Delta \omega, \Delta \phi)$. We can maximize the distance for a given frequency difference by differentiating with respect to phase difference $\Delta \phi$ \cite{budzynski2008applications}.  In Fig. \ref{fig:BBDSNR}, we show the variation of maximized BBD for different signal to noise ratios ($\varrho_1,\varrho_2$), for a fixed frequency difference $\Delta \omega=0.01$.  The intensity map shows different bands which can be used for setting the threshold for signal separation. 

Detecting signal of known form involves minimizing the distance measure over the parameter space of the signal. A threshold on the maximum ``distance'' between the signal and filter can be put so that a detection is said to occur whenever the measures fall within this threshold. Based on a series of tests,  Receiver Operating Characteristic (ROC) curves can be drawn to study the effectiveness of the distance measure in signal detection. We leave such details for future work.

\begin{figure}[!ht]
\includegraphics[scale=0.4]{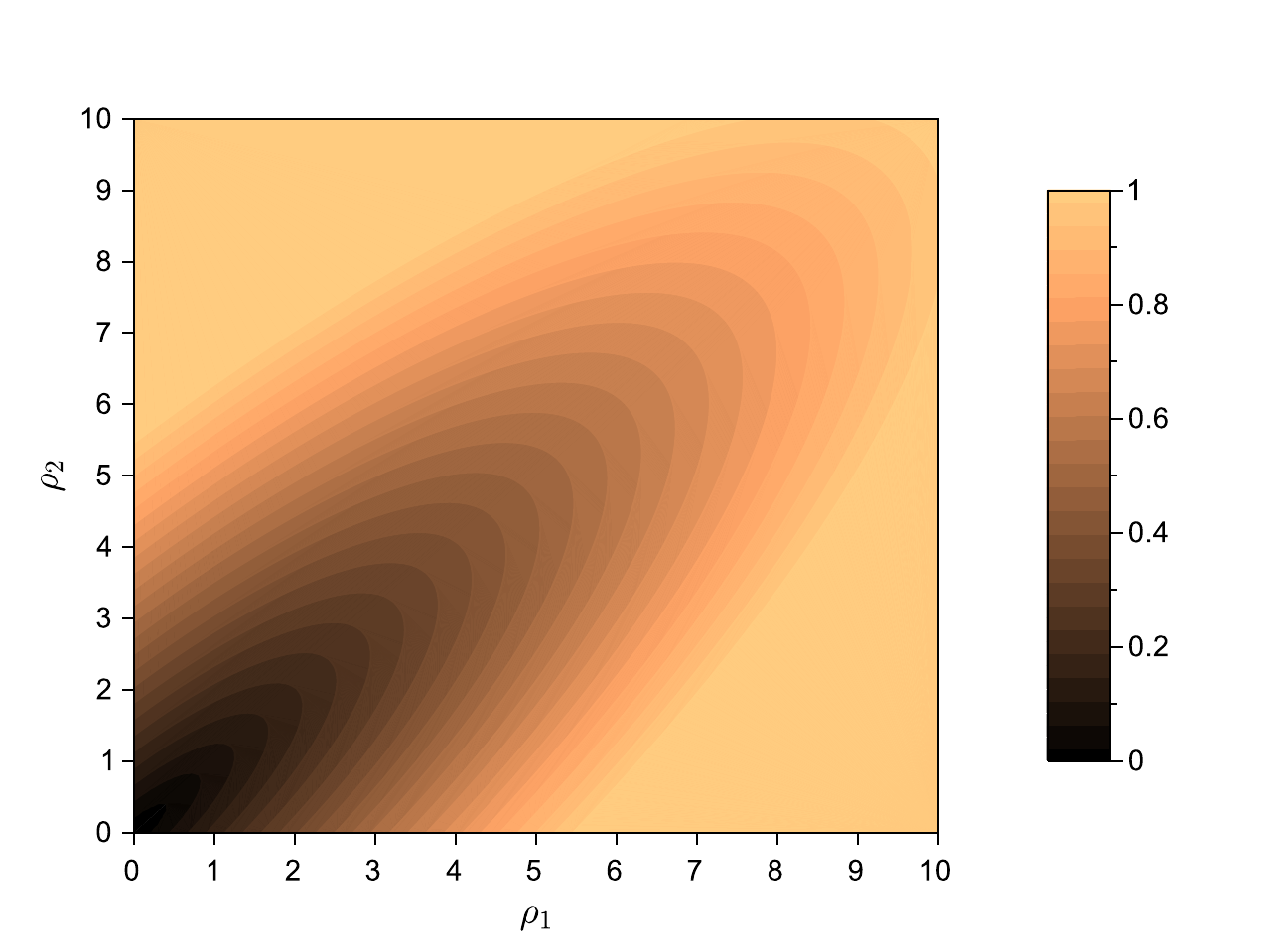}
\caption{BBD with different signal to noise ratio for a fixed . We have set  $T=100$ and $\Delta \omega=0.01$.}
\label{fig:BBDSNR}
\end{figure}

%\subsubsection{Amplitude and Frequency difference}
%
%When both amplitude and frequency differ, we can analyze the behavior near the origin by Taylor expansion. 
%
%\begin{equation}
%B(\Delta \omega, \Delta A) \sim \frac{T}{16S_0\ln 2}\left[(\Delta A)^2+ \frac{A^2T^2}{3}(\Delta \omega)^2\right]
%\end{equation}�

\section{\uppercase{Summary and Outlook}}
In this work we have introduced a new family of bounded divergence measures based on the Bhattacharyya distance, called bounded Bhattacharyya distance measures.
We have shown that it belongs to the class of generalized f-divergences and inherits all its properties, such as those relating Fishers Information and curvature metric.  We have discussed several special cases of our measure, in particular squared Hellinger distance, and studied relation  with other measures such as Jensen-Shannon divergence. We have also extended the Bradt Karlin theorem on error probabilities to BBD measure.  Tight bounds on Bayes error probabilities can be put by using properties of Bhattacharyya coefficient. 

Although many bounded divergence measures have been studied and used in various  applications, no single measure is useful in all types of problems studied. Here we have illustrated an application to signal detection problem by considering ``distance'' between monochromatic signal and filter buried in white Gaussian noise with differing frequency or amplitude, and comparing it to Fishers Information and Kullback-Leibler divergence. 

A detailed study with chirp like signal and colored noise occurring in Gravitational wave detection will be taken up in a future study.  Although our measures have a tunable parameter $\alpha$, here we have focused on a special case with $\alpha =2 $. In many practical applications where extremum values are desired such as minimal error, minimal false acceptance/rejection  ratio etc, exploring the BBD measure by varying $\alpha$ may be desirable.  Further, the utility of BBD measures is to be explored in parameter estimation based on minimal disparity estimators and Divergence information criterion in Bayesian model selection \cite{basu1994minimum}. However, since the focus of the current paper is introducing a new measure and studying its basic properties, we leave such applications to statistical inference and data processing to future studies.

\section*{\uppercase{Acknowledgements}}
One of us (S.J) thanks Rahul Kulkarni for insightful discussions, Anand Sengupta for discussions on application to signal detection, and acknowledge the financial support in part by grants DMR-0705152 and DMR-1005417 from the US National Science Foundation. M.S. would like to thank the Penn State Electrical Engineering Department for support.

\vfill
\bibliographystyle{apalike}
{\small
\bibliography{Papers}}

\section*{\uppercase{Appendix}}
\label{app}

{\bf BBD measures of some common distributions}

Here we provide explicit expressions for BBD $B_2$, for some common distributions. For brevity we denote $\zeta\equiv B_2$.  
\begin{itemize}
\item {\bf Binomial} :\\
\begin{center}
 $P(k) = \binom{n}{k}p^k(1-p)^{n-k},$  $Q(k) = \binom{n}{k} q^k(1-q)^{n-k}. $ 
\end{center}

\begin{equation}
\zeta_{bin}(P,Q)=-\log_2 \left( \frac{1+[\sqrt{pq}+\sqrt{(1-p)(1-q)}]^n }{2}\right) .  
\end{equation} 
 
\item {\bf Poisson } : \\
\begin{center}
$P(k) =\frac{\lambda_p^ke^{-\lambda_p}}{k!},$ $Q(k) =\frac{\lambda_q^ke^{-\lambda_q}}{k!}.$ 
\end{center}

\begin{equation}
\zeta_{poisson}(P,Q)=-\log_2 \left(\frac{1+e^{-(\sqrt{\lambda_p}-\sqrt{\lambda_q})^2/2}}{2}\right) .
\end{equation} 

\item {\bf Gaussian} : 
\begin{eqnarray}
P(x) &=&\frac{1}{\sqrt{2\pi}\sigma_p}\exp\left(-\frac{(x-x_p)^2}{2\sigma_p^2}\right), \nonumber \\
Q(x) &=&\frac{1}{\sqrt{2\pi}\sigma_q}\exp\left(-\frac{(x-x_q)^2}{2\sigma_q^2}\right) .\nonumber 
\end{eqnarray}

\begin{equation}
\zeta_{Gauss}(P,Q) =1-\log_2\Big[1+\frac{{2\sigma_p \sigma_q}}{\sigma_p^2+\sigma_q^2}   \exp\left(-\frac{(x_p-x_q)^2}{4(\sigma_p^2+\sigma_q^2)}  \right) \Big]. 
\end{equation}

\item {\bf Exponential} : $P(x)=\lambda_p e^{-\lambda_p x} $, $Q(x)=\lambda_q e^{-\lambda_q x} . $
\begin{equation}
\zeta_{exp}(P,Q)=-\log_2\left[\frac{(\sqrt{\lambda_p}+\sqrt{\lambda_q})^2}{2(\lambda_p+\lambda_q)} \right] .
\end{equation} 

\item {\bf Pareto} : Assuming the same cut off $x_m$,

\begin{equation}
P(x)=\begin{cases}
\alpha_p \frac{x_m^{\alpha_p}}{x^{\alpha_p+1}} & \text{for } x \ge x_m \\
0& \text{for} ~ x<x_m,
\end{cases}
\end{equation}

\begin{equation}
Q(x)=\begin{cases}
\alpha_q \frac{x_m^{\alpha_q}}{x^{\alpha_q+1}} & \text{for} ~~x\ge x_m\\
0& ~ \text{for } ~ x<x_m .
\end{cases}
\end{equation}

\begin{equation}
 \zeta_{pareto}(P,Q)=-\log_2\left[\frac{(\sqrt{\alpha_p}+\sqrt{\alpha_q})^2}{2(\alpha_p+\alpha_q)} \right].
\end{equation} 
\end{itemize}

\vfill
\end{document}